\documentclass[reqno]{amsart}
\usepackage[pxpazo, numberbasedon=subsection]{zzzams}
\usepackage{bbm}

\usepackage[nomarkers, nolists]{endfloat} % Put figures to the end

\DeclareMathOperator{\Bh}{\mathbbm{h}}

\DeclareMathOperator{\Sat}{Sat}

\newcommand{\ubar}[1]{\underline{#1}}
\DeclareMathOperator{\cl}{\mathbf{cl}}
\DeclareMathOperator{\AV}{AV}

\newcommand{\xeq}[1]{\overset{\mathrm{#1}}{=\joinrel=}}

\newcommand{\laurent}[1]{(\!({\ensuremath{#1}})\!)}
\newcommand{\series}[1]{[\![{\ensuremath{#1}}]\!]}

\title{Associated varieties of simple affine vertex algebras at rational levels}

\author{Peng Shan$^{1,2}$}
\author{Wenbin Yan$^1$} 
\author{Qixian Zhao$^1$}

\address{\scriptsize{$^1$} Yau Mathematical Sciences Center, Tsinghua University, Beijing, 100084, China}
\address{\scriptsize{$^2$} Department of Mathematical Sciences, Tsinghua University, Beijing, 100084, China}

\dedicatory{Dedicated to George Lusztig on the occasion of his 80th birthday}

\begin{document}

\begin{abstract}
	We present a conjecture for associated varieties of simple affine vertex algebras $L_k(\fg)$ attached to a simple Lie algebra $\fg$ of simply-laced type and any rational level $k$ greater than the critical level. The key new ingredient compared to the integral case is the covering duality map introduced by Gao-Liu-Lo-Shahidi. We provide evidence for the conjecture.
\end{abstract}

\maketitle

\tableofcontents

%============
\section{Introduction}

Let $\fg$ be a simple Lie algebra over $\BC$ and let $k \in \BQ$. Let $L_k = L_k(\fg)$ be the unique simple quotient of the universal affine vertex algebra $V^k(\fg)$ at level $k$. It comes equipped with a geometric invariant called \textit{associated variety} $X_{L_k(\fg)}$, see \S \ref{subsec:AV-def}, which is a $G$-stable conic Poisson subvariety in $\fg$ (we identify $\fg^*$ and $\fg$ using an invariant form).

Let $\check \Bh$ be the dual Coxeter number of $\fg$, and let $k + \check \Bh = \frac mu$ in lowest terms. Then $X_{L_k}$ is equal to $\fg$ if $m < 0$, to the nilpotent cone $\cN$ in $\fg$ if $m = 0$, and to the closure of a nilpotent orbit if $k$ is admissible \cite{Arakawa:C2} (when $\fg$ is simply-laced, $k$ admissible means $m \ge \check \Bh$). In the remaining cases, the behavior of $X_{L_k(\fg)}$ is very interesting but mysterious, and is currently known only in sporadic examples and subfamilies \cite{Arakawa-Moreau:Omin, Arakawa-Moreau:sheets, Arakawa-Moreau:irred, AFK, Jiang-Song, ADFLM}. The purpose of this paper is to propose a general conjectural description of $X_{L_k(\fg)}$ for rational $k> - \check \Bh$ when $\fg$ is of simply-laced types, and to provide evidence for the conjecture.

At non-admissible levels $k$, the $L_k(\fg)$'s are important from physics point of view. For example, the $L_k(\fg)$'s sometimes appear in the 4d/VOA correspondences \cite{BLLPRvR} attached to certain 4d $\cN =2$ superconformal field theories $\cT$, and it is conjectured that their associated varieties agree with the Higgs branches of the physical theories \cite{Beem-Rasterlli:VOA-Higgs}. % The vertex algebras arising this context must be \textit{quasi-lisse} \cite{Arakawa-Kawasetsu}, meaning that their associated varieties must have finitely many symplectic leaves. In the case of $L_k(\fg)$, it is quasi-lisse if and only if $X_{L_k(\fg)}$ is contained in the nilpotent cone $\cN$. It is therefore important to determine when this happens.
The associated varieties $X_{L_k(\fg)}$ are also important in the representation theory of the affine Lie algebra $\hat \fg$. In the finite dimensional setting, associated varieties of simple highest weight modules have close connections with primitive ideals, Springer fibers, and Kazhdan-Luszig cells, see for example \cite{Lusztig:char,Barbasch-Vogan:unipotent}. When $k$ is non-admissible, the $\hat \fg$-module $L_k(\fg) = L(k \Lambda_0)$ is one of the first examples of a $G\series{t}$-equivariant simple object in a singular block of category $\cO$. It is therefore natural to expect that $X_{L_k(\fg)}$ should fit into a picture involving cells in the affine Weyl group. In the prequel \cite{SYZ:cl} to this paper, we initiated the first attempts in building such a picture and proposed a conjectural description for $X_{L_k(\fg)}$ when $\fg$ is simply-laced and $k$ is a non-admissible integer.

Let us review the precise statements. Let $\fg$ be of simply-laced type, and let $k + \check \Bh = m \ge 1$ be an integer. The conjecture in \textit{op. cit.} is phrased in terms of a nilpotent orbit $\check \BO(m)$ in the Langlands dual Lie algebra $\check \fg$ attached to $m$. The definition of $\check \BO(m)$ and its role in Langlands duality are recalled in \S \ref{subsec:cl}. Write $\check \BO(m)$ as a saturation $\Sat_{\check L}^{\check G} \BO_{\check L}$ from a Bala-Carter Levi $\check \fl$ of $\check \BO(m)$ (that is, a minimal Levi subalgebra among those intersecting $\check \BO(m)$ nontrivially). Let $\fp = \fl \oplus \fu$ be a parabolic subalgebra containing the Langlands dual $\fl$ of $\check \fl$, and let $\fz(\fl)$ be the center of $\fl$. Let $\bd$ (resp. $\bd_L$) denote the Barbasch-Vogan-Lusztig-Spaltenstein duality map from nilpotent orbits of $\check \fg$ (resp. $\check \fl$) to those in $\fg$ (resp. $\fl$) \cite{Lusztig:sp-1,Barbasch-Vogan:unipotent}.

\begin{conjecture}[{\cite[Conjecture 1.0.2]{SYZ:cl}}]\label{conj:int}
	In the above setup,
	\begin{equation*}
		X_{L_k(\fg)} = \overline{\Ad G \cdot (\bd_L \BO_{\check L} \times \fz(\fl) \times \fu)}.
	\end{equation*}
	In particular, 
	\begin{equation*}
		X_{L_k(\fg)} \cap \cN = \overline{\bd \check \BO(m)},
	\end{equation*}
	and $X_{L_k(\fg)} \subseteq \cN$ if and only if $\check \BO(m)$ is distinguished.
\end{conjecture}

We now turn to rational levels. As before, for $k \in \BQ$ above critical level we write $k + \check \Bh = \frac mu$ in lowest terms with $m \ge 1$. We replace the Langlands dual $\check \fg$ by the \textit{metaplectic dual} $\check \fg^{(u)}$, see \S \ref{subsec:mdual}. This is a Lie algebra  isomorphic to $\check \fg$ when $\fg$ is simply-laced. Again, we let $\check \BO(m)^{(u)} = \Sat_{\check L^{(u)}}^{\check G^{(u)}} \BO_{\check L^{(u)}}$ be the orbit in $\check \fg^{(u)}$ attached to $m$, written as a saturation from its Bala-Carter Levi $\check L^{(u)}$. Finally, we replace the duality map $\bd$ by the \textit{covering duality map} defined by Gao-Liu-Lo-Shahidi \cite{Gao-Liu-Lo-Shahidi}
\begin{equation*}
	\bd^{(u)}: \ubar{\check \cN}^{(u)} \aro \ubar \cN
\end{equation*}
from nilpotent orbits in $\check \fg^{(u)}$ to those in $\fg$ (see \S \ref{subsec:mdual}). The Levi version of this map $\bd_L^{(u)}$ sends orbits in $\check \fl^{(u)}$ to orbits in a Levi $\fl$ of $\fg$. The following is the main conjecture of this paper.

\begin{conjecture}
	In the above situation, 
	\begin{equation*}
		X_{L_k(\fg)} = \overline{\Ad G \cdot (\bd_L^{(u)} \BO_{\check L^{(u)}} \times \fz(\fl) \times \fu)}.
	\end{equation*}
	In particular, 
	\begin{equation*}
		X_{L_k(\fg)} \cap \cN = \overline{\bd^{(u)} \check \BO(m)^{(u)}},
	\end{equation*}
	and $X_{L_k(\fg)} \subseteq \cN$ if and only if $\check \BO(m)^{(u)}$ is distinguished.
\end{conjecture}

This appears as Conjecture \ref{conj:main} below. It is worth noticing that $\check \BO(m)^{(u)}$ is distinguished if and only if $\check \BO(m)$ is. As a result, the quasi-lisse-ness of $L_k(\fg)$, which is equivalent to whether $X_{L_k(\fg)} \subseteq \cN$ (see \ref{subsec:AV-def}), depends only on $m$, not on $u$.

The appearance of the metaplectic dual and the covering duality map deserves some justification. For this, it is helpful to briefly review the work \cite{Gao-Liu-Lo-Shahidi}. Let $G$ be a $p$-adic group, and let $\bar G^{(u)}$ be a $u$-fold cover of $G$. For any irreducible genuine smooth representation $\pi$ of $\bar G^{(u)}(\BQ_p)$, its geometric wavefront set $\operatorname{WF}^{geom}(\pi)$ is a subvariety in the nilpotent cone $\cN(\fg(\bar \BQ_p))$, which can also be identified with a subvariety in $\cN(\fg(\BC))$. The main goal of \textit{op. cit.} is to formulate a conjectural upper bound for $\operatorname{WF}^{geom}(\pi)$ in terms of the hypothetical L-parameter $\phi_{\operatorname{AZ}(\pi)}: \operatorname{WD}_{\BQ_p} \to {}^L \bar G^{(u)}$ of $\operatorname{AZ}(\pi)$. Here $\operatorname{AZ}(\pi)$ is the Aubert-Zelevinskii dual of $\pi$, $\operatorname{WD}_{\BQ_p}$ is the Weil-Deligne group, and ${}^L \bar G^{(u)}$ is the L-group attached to $\bar G^{(u)}$. By the theorem of Jacobson-Morozov, the restriction of $\phi_{\operatorname{AZ}(\pi)}$ to the $\SL_2$ inside $\operatorname{WD}_{\BQ_p}$ produces a nilpotent orbit $\BO(\phi_{\operatorname{AZ}(\pi)})$ in the metaplectic dual $\check \fg^{(u)}$. Their conjecture then says 
\begin{equation*}
	\operatorname{WF}^{geom}(\pi) \subseteq \overline{\bd^{(u)} \,\BO(\phi_{\operatorname{AZ}(\pi)}) }.
\end{equation*}

Under the analogy between the representation theories of $G(\BQ_p)$ and $G(\BC\laurent{t})$, genuine representations $\pi$ of the $u$-fold cover $\bar G^{(u)}(\BQ_p)$ correspond to modules $M$ over the affine Lie algebra $\fg_{aff}$ of a certain rational level $k$, with $k$ having denominator $u$. The geometric wavefront set $\operatorname{WF}^{geom}(\pi)$ then parallels the associated variety $X_M$ of $M$, and for $M = L(k \Lambda_0) = L_k(\fg)$, the orbit $\check \BO(m)^{(u)}$ serves as an analog to $\BO(\phi_{\operatorname{AZ}(\pi)})$. As a first concrete example, $L_k(\fg)$ at admissible levels are direct affine analogs of the theta representations $\Theta$ of $\bar G^{(u)}$, and we have literally
\begin{equation*}
	\operatorname{WF}^{geom}(\Theta) = \overline{\bd^{(u)} \check \BO_{reg}^{(u)}} = \overline{\bd^{(u)} \check \BO(m)^{(u)}} = \overline{\BO(u)} = X_{L_k(\fg)}
\end{equation*}
when $p$ is large enough, see \cite[Theorem 1.1]{Karasiewicz-Okada-Wang} and \cite{Arakawa:C2}. This example will be revisited in \ref{cls:adm}.

%The paper is organized as follows. We 

%-------
\subsection{Acknowledgement}

We would like to thank Tomoyuki Arakawa, Baohua Fu, Fan Gao, Cuibo Jiang, Anne Moreau, and Shilin Yu for valuable discussions. Q. Zhao further thanks Shilin Yu and Tianyuan Mathematical Center in Southeast China for organizing and inviting him to the conference \textit{Representation Theory of Lie Groups and Related Topics} at Xiamen, China, 2025 where some of the inspirations for this work originated. 

This work is supported by National Key R\&D Program of China (No. 2025YFA1017400).
PS is also supported by NSFC Grant 12225108 and by the New Cornerstone Science Foundation through the Xplorer Prize.

%===========
\section{Preliminaries}

%---------
\subsection{Associated varieties}\label{subsec:AV-def}

\begin{clause}[Lie theoretic notations]
	We let $\fg$ be a simple Lie algebra over $\BC$ with a Cartan subalgebra $\fh$ and a Borel $\fb \supseteq \fh$. We let $\Bh$ (resp. $\check \Bh$) be the Coxeter number (resp. dual Coxeter number) of $\fg$. Write $\rho$ for the half sum of positive roots with respect to the Borel $\fb$. The \textit{affine Lie algebra} is a central extension
	\begin{equation*}
		0 \aro \BC K \aro \fg_{aff} \aro \fg\laurent{t} \aro 0.
	\end{equation*}
	Such extensions are determined by a nondegenerate invariant symmetric bilinear form $B(-,-)$ on $\fg$, which we fix to be the unique one such that the induced form on $\fh^*$ satisfies $B(\theta, \theta) = 2$ for the highest root $\theta$. We fix the Cartan subalgebra $\fh_{aff} = \fh \oplus \BC K$ and write its dual as $\fh_{aff}^* = \fh^* \oplus \BC \Lambda_0$ where $\Lambda_0$ is dual to $K$. 
	
	We write $W_{aff}$ for the affine Weyl group, with a set of simple reflections determined by the Borel $\fb$. It acts naturally on $\fh_{aff}^*$. We write $\cN$ for the nilpotent cone in $\fg$, and write $\ubar \cN$ for the set of $\Ad G$-orbits in $\cN$.
\end{clause}

\begin{clause}[Affine vertex algebras]
	For the definition of vertex algebras and its properties, we refer readers to \cite{Kac:VA, Frenkel-Ben-Zvi, Arakawa:W-alg-notes}. Given a vertex algebra $V$ and $a \in V$, we write $Y(a,z) = \sum_{n \in \BZ} a_{(n)} z^{-n-1} \in \End(V) \series{z,z\inv}$ for the corresponding field.
	
	For a number $k \in \BC$, let
	\begin{equation*}
		V^k = V^k(\fg) = \cU(\fg_{aff}) \dotimes_{\cU(\fg\series{t}) \oplus \,\BC K} \BC_k,
	\end{equation*}
	where $\fg\series{t}$ acts trivially on $\BC_k$ and $K$ acts on $\BC_k$ by $k$. This is a highest weight module of $\fg_{aff}$ with highest weight $k \Lambda_0$ and it has the structure of a vertex algebra (see, for example, \cite{Frenkel-Zhu} or \cite[\textsection 2.4]{Frenkel-Ben-Zvi}). 
	The unique irreducible quotient $L(k \Lambda_0)$ of $V^k$ as a $\fg_{aff}$-module inherits a structure of a vertex algebra from the one on $V^k$. We denote the resulting vertex algebra by 
	\begin{equation*}
		L_k = L_k(\fg).
	\end{equation*}
	
	For a vertex algebra $V$, \textit{Zhu's $C_2$-algebra of $V$} is the ring
	\begin{equation*}
		R_V = V/ \{a_{(-2)} b \mid a,b \in V\}
	\end{equation*}
	with multiplication given by $a \cdot b = a_{(-1)} b$ and a Poisson bracket given by $\{a,b\} = a_{(0)}b$ \cite{Zhu:thesis}. The subspace $\{a_{(-2)} b \mid a,b \in V\}$ is usually denoted by $C_2(V)$, and it coincides with the first piece $F^1 V$ in the filtration on $V$ defined by Li \cite{Li:abVA}.	The \textbf{associated variety} of $V$ is Poisson variety
	\begin{equation*}
		X_V := (\Spec R_V )_{red}.
	\end{equation*}
	We say $V$ is \textit{quasi-lisse} if the Poisson variety $X_V$ has finitely many symplectic leaves.
	
	We have $X_{V^k} = \fg^*$ which can (and will) be identified with $\fg$ using the form $B(-,-)$ fixed above. Then $X_{L_k}$ is a closed $G$-stable conic Poisson subvariety inside $X_{V^k} = \fg$, and $L_k$ is quasi-lisse if and only if $X_{L_k} \subseteq \cN$.
\end{clause}

%-----------
\subsection{Cyclotomic levels}\label{subsec:cl}

\begin{definition}[{\cite[Definition 2.1.2]{SYZ:cl}}]
	The \textbf{cyclotomic level map} is defined as
	\begin{equation*}
		\cl_n: \cN \aro \BZ_{\ge 1},\quad
		e \mapsto \min\{ m \mid (\ad e)^{2m} = 0 \}.
	\end{equation*}
	It clearly descend to a map on orbits
	\begin{equation*}
		\cl_n: \ubar \cN \aro \BZ_{\ge 1}.
	\end{equation*}
	Alternatively, for any $e \in \cN$, let $\fl$ be a minimal Levi containing $e$ (i.e. a Bala-Carter Levi of $e$), let $\{h,e,f\} \subset \fl$ be an $\fsl_2$-triple, and let $2a$ be the largest $\ad h$-weight on $\fl$. Then 
	\begin{equation*}
		\cl_n: e \mapsto a+1.
	\end{equation*}
	The equivalence of the two definitions is proven in \cite[Lemma 2.1.4]{SYZ:cl}.
\end{definition}

The following key property of $\cl_n$ is taken from \cite[Theorem 2.1.6]{SYZ:cl}; see also \cite{VIGRE}.

\begin{theorem}
	For any positive integer $m$, there exists a unique orbit $\BO(m) \in \ubar \cN$ so that
	\begin{equation*}
		\cl_n\inv([1,m]) = \overline{\BO(m)}.
	\end{equation*}
\end{theorem}

The orbit $\BO(m)$ and the value of $\cl_n$ on each orbit is described explicitly in \cite[\S 2.1]{SYZ:cl}.

Now let $\check \BO(m) \in \ubar{\check \cN}$ be the orbit corresponding to $m$ in the Langlands dual Lie algebra $\check \fg$. These orbits are especially important in the Langlands correspondence involving $L_k(\fg)$ for integral $k$ above critical level, which is one of the conceptual reasons for their appearance in Conjecture \ref{conj:int}. To be more precise, consider the affine weight $m \Lambda_0 + \rho$. When $m = k + \check \Bh$, the weight $m \Lambda_0 + \rho$ and the highest weight $k \Lambda_0$ of $L_k(\fg)$ differ by an affine rho-shift. Let
\begin{itemize}
	\item $\xi_m$ be the unique dominant $W_{aff}$-translate of $m \Lambda_0 + \rho$ under the usual (non-dot) action, let
	\item $w_m \in W_{aff}$ be the longest element stabilizing the weight $\xi_m$, and let
	\item $\ubar \bc(w_m) \subset W_{aff}$ be the two-sided cell in $W_{aff}$ containing $w_m$.
\end{itemize}
Now recall Lusztig's bijection
\begin{equation*}
	\ubar{\check \cN} \bijects \{\text{two-sided cells in } W_{aff}\}
\end{equation*}
established in \cite[Theorem 4.8]{Lusztig:aff-cells-4}. The following result is proven in \cite[Theorem 4.2.1]{SYZ:cl}.

\begin{theorem}
	Let $\fg$ be of simply-laced type. For each integer $m \ge 1$, Lusztig's bijection sends $\check \BO(m)$ to $\ubar \bc(w_m)$.
\end{theorem}

%--------
\subsection{Metaplectic dual and covering duality}\label{subsec:mdual}

Let $G$ be the simply connected group with Lie algebra $\fg$. Fix a maximal torus $T$ in $G$ with Lie algebra $\fh$.

To define metaplectic dual groups, we need fix a positive integer $u$, which is supposed to be the degree of a metaplectic cover, and a central extension of sheaves of groups
\begin{equation*}
	1 \aro K_2 \aro \widetilde G \aro G \aro 1
\end{equation*}
where $K_2$ denotes the sheaf taking the second $K$-group. By \cite{Brylinski-Deligne}, such extensions are classified by Weyl group invariant quadratic forms $Q: X_*(T) \to \BZ$, or equivalently Weyl group invariant bilinear forms $B:X_*(T) \times X_*(T) \to \BZ$ ($Q$ and $B$ are related by $Q(x) = \frac12 B(x,x)$). Such forms are unique up to an integer constant, and they also classify, by \cite{Finkelberg-Lysenko}, central extensions
\begin{equation*}
	1 \aro \BG_m \aro \widetilde{LG} \aro LG \aro 1
\end{equation*}
of the loop group $LG$ by $\BG_m$.

From now on, we fix $Q$ (or $B$) to be the smallest one, i.e. one such that $Q(\check \alpha) = 1$ (or $B(\check \alpha, \check \alpha) = 2$) for any short coroot $\check \alpha$. 

\begin{definition}
	The \textbf{metaplectic dual group} attached to $u$ (and the fixed $Q$) is a reductive group $\check G^{(u)}$ with maximal torus $\check T^{(u)}$ whose root datum $(X^*(\check T^{(u)}), \Phi(\check G^{(u)}), X_*(\check T^{(u)}), \check \Phi(\check G^{(u)}))$ is given by:
	\begin{align*}
		X^*(\check T^{(u)}) &= \big\{ \lambda \in X_*(T) \mid \forany \mu \in X_*(T), B(\lambda,\mu) \in u \BZ \big\}\\
		\Phi(\check G^{(u)}) &= \big\{ \check \alpha_{Q,u} := u_\alpha \check \alpha \mid \check \alpha \in \check \Phi(G) \big\}\\
		X_*(\check T^{(u)}) &= \text{sublattice in } X^*(T) \otimes \BQ \text{ that is dual to } X^*(\check T^{(u)})\\
		\check \Phi(\check G^{(u)}) &= \big\{ \alpha_{Q,u} := u_\alpha\inv \alpha \mid  \alpha \in \Phi(G) \big\}
	\end{align*}
	where $(X^*(T), \Phi(G), X_*(T), \check \Phi(G))$ is the root datum of $G$ and $u_\alpha = u/\gcd(u, Q(\check \alpha))$.
\end{definition}

Let $\check G$ be the (usual) Langlands dual group of $G$. Since $X^*(T)$ pairs integrally with elements in $X_*(T)$, in particular with those in $X^*(\check T^{(u)})$, $X^*(T)$ is a sublattice of $X_*(\check T^{(u)})$ of full rank. Using this we may identify
\begin{equation*}
	\check \ft^{(u)} = X_*(\check T^{(u)}) \dotimes_\BZ \BC = X^*(T) \dotimes_\BZ \BC = \check \ft^* = \ft.
\end{equation*}
Then the fundamental coweight $\omega(\check \alpha_{Q,u})$ associated to $\check \alpha_{Q,u}$ is identified with $u_\alpha\inv \omega(\check \alpha)$.

Write $\ubar{\check \cN}^{(u)}$ for the set of nilpotent orbits in $\check \fg^{(u)}$.

\begin{definition}[{\cite[\S 2.1]{Gao-Liu-Lo-Shahidi}}]
	The \textbf{covering duality map}
	\begin{equation*}
		\bd^{(u)}: \ubar{\check \cN}^{(u)} \aro \ubar \cN
	\end{equation*}
	is defined as follows. For an orbit $\check \BO^{(u)} \in \ubar{\check \cN}^{(u)}$, let $\{e,h,f\} \subset \check \fg^{(u)}$ be an $\fsl_2$-triple with $h \in \check \ft^{(u)}$ dominant. Suppose the identification $\check \ft^{(u)} = \ft$ sends $h \mapsto h^{(u)}$. Then we set $\bd^{(u)} \check \BO^{(u)} = \BO$ to be the unique open orbit $\BO$ in the associated variety of the maximal primitive ideal of $\cU(\fg)$ with infinitesimal character $h^{(u)}/2$. 
\end{definition}

This map generalizes the metaplectic duality map defined in \cite{BMSZ:metap-dual}. The value of $\bd^{(u)}$ in each type is computed explicitly in \textit{op. cit.} When $u=1$, we recover the duality of Barbasch-Vogan-Lusztig-Spaltenstein \cite{Lusztig:sp-1,Barbasch-Vogan:unipotent}.

\begin{remark}[Distinction in type $D$]\label{rmk:Spin-vs-SO}
	Since we assume $G$ is simply-connected, in type $D$ our $G$ is the Spin group. On the contrary, \cite[\S 3]{Gao-Liu-Lo-Shahidi} takes $G$ to be the special orthogonal group. As a result, their combinatorial formula for $\bd^{(u)}$ is based on $Q(\check \alpha) =2$ rather than $Q(\check \alpha)=1$. In the notation there, when $u$ is even, their formula for $\bd^{(u)}(-)$ is $\big( \bd^{(u/2)}_{com,A}(-)\big)_D$, while we use $\big( \bd^{(u)}_{com,A}(-)\big)_D$.
\end{remark}

\begin{clause}[Simply-laced types]\label{cls:du-ADE}
	Suppose $\fg$ is simply-laced. Then under our choice of $Q$, we have $u_\alpha = u$ for all root $\alpha$. From the description of $\Phi(\check G^{(u)})$, it is clear that $\check \fg^{(u)}$ and $\check \fg$ have the same type. As a result, we have naturally a bijection $\ubar{\check \cN}^{(u)} \bij \ubar{\check \cN}$, $\check \BO^{(u)} \mapsto \check \BO$ under which the weighted Dynkin diagrams of $\check \BO^{(u)}$ and $\check \BO$ are identified and the neutral elements in their respective $\fsl_2$-triples satisfy
	\begin{equation*}
		h^{(u)} = h/u.
	\end{equation*}
	Consequently, 
	\begin{equation*}
		\overline{\bd^{(u)} \check \BO^{(u)} } = \AV( \cU(\fg)/J_{h/2u})
	\end{equation*}
	where $\AV$ denotes the associated variety, and $J_{h/2u}$ is the maximal primitive ideal with infinitesimal character $h/2u$. 
\end{clause}

%---------
\subsection{Sheets in $\fg$}

\begin{definition}
	For a Levi subalgebra $\fl \subset \fg$ and a nilpotent orbit $\BO_L \subset \fl$, the \textit{Jordan class} (or the \textit{decomposition class}) attached to the pair $(\fl, \BO_L)$ is the subset in $\fg$ given by
	\begin{equation*}
		J_G(\fl, \BO_L) = \Ad G \cdot (\fz(\fl)^{reg} + \BO_L)
	\end{equation*}
	where $X^{reg}$ takes elements in $X$ whose centralizer has minimum possible dimension among elements in $X$. This is locally closed in $\fg$ \cite{Borho-Kraft:deform}. Equivalently, $J_G(\fl, \BO_L)$ consists of all elements $x \in \fg$ such that, if $x = h+e$ is its Jordan decomposition, then the pair $(\fz(\fz_\fg(h))^{reg}, \Ad Z_G(h) \cdot e)$ is $\Ad G$-conjugate to $(\fl, \BO_L)$. In this case we also write $J_G(\fl, \BO_L) = J_G(x)$.
	
	The \textbf{(generalized) sheet} attached to $(\fl, \BO_L)$ is 
	\begin{equation*}
		\cS(\fl, \BO_L) = \big( \overline{J_G(\fl, \BO_L)} \big)^{reg}.
	\end{equation*}
	When $\BO_L$ is a rigid orbit in $\fl$ (i.e. when $\BO_L$ cannot be obtained from Lusztig-Spaltenstein induction from an orbit in a proper Levi in $\fl$, see \cite{Lusztig-Spaltenstein:induced} or \cite[\S 2]{Borho:sheets}), $\cS(\fl, \BO_L)$ is a \textit{sheet} in the sense of Borho-Kraft \cite{Borho-Kraft:deform}. The closure of a (generalized) sheet is
	\begin{equation}\label{eqn:sheet-closure}
		\overline{\cS(\fl, \BO_L)} = \overline{J_G(\fl, \BO_L)} = \Ad G \cdot \big( \fz(\fl) + \overline{\BO_L} + \fu\big)
	\end{equation}
	where $\fp = \fl \oplus \fu$ is any parabolic in $\fg$ containing $\fl$ as its Levi factor and $\fu$ is the radical of $\fp$, see \cite[\S 2.5 Lemma]{Borho:sheets}. We have
	\begin{equation}\label{eqn:dim-of-sheet}
		\dim \overline{\cS(\fl, \BO_L)} = \dim \fz(\fl) + \dim \Ind_L^G \BO_L,
	\end{equation}
	see \cite[Proposition 39.2.9]{Tauvel-Yu}.
\end{definition}

\begin{lemma}\label{lem:sheets-cap-N-ss}
	We have
	\begin{enumerate}
		\item $\overline{\cS(\fl, \BO_L)} \cap \cN = \overline{\Ind_L^G \BO_L}$, where $\Ind_L^G \BO_L$ denotes the Lusztig-Spaltenstein induction of nilpotent orbits.
		
		\item $\overline{\cS(\fl, \BO_L)} \cap \fg_{ss} = \Ad G \cdot \fz(\fl)$, where $\fg_{ss}$ denotes the set of semisimple elements in $\fg$.
	\end{enumerate}
\end{lemma}

\begin{proof}
	Part (1) follows from \cite[\S 3.2]{Borho:sheets}. For part (2), first it is clear from (\ref{eqn:sheet-closure}) that $\Ad G \cdot \fz(\fl)$ is contained in $\overline{\cS(\fl, \BO_L)}$. Conversely, by \cite[\S 2.1]{Borho:sheets} any element in (\ref{eqn:sheet-closure}) has its semisimple part contained in $\Ad G \cdot \fz(\fl)$.
\end{proof}

%\begin{clause}[Sheets and Slodowy slices]
%	We now collect a few facts between (generalized) sheets and Slodowy slices. They will be used in \S\ref{subsec:evid} when presenting evidences related to $\cW$-algebras.
%	
%	????
%\end{clause}

%=============
\section{The main conjecture and evidence}

%--------
\subsection{Statement of the conjecture}

Throughout this section, let $\fg$ be simple of simply-laced type, and let $k \in \BQ$ be such that
\begin{equation*}
	k + \check \Bh = \frac mu \text{ with } m \ge 1 \text{ and } \operatorname{gcd}(m,u) = 1.
\end{equation*}
Recall from \ref{cls:du-ADE} that if $\fg$ is simply-laced, $\check \fg^{(u)}$ and $\check \fg$ have the same type. As a result, we have the orbit $\check \BO(m)^{(u)} \in \ubar{\check \cN}^{(u)}$. Write $\check \BO(m)^{(u)} = \Sat_{\check L^{(u)}}^{\check G^{(u)}} \BO_{\check L^{(u)}}$ where $\check L^{(u)}$ is the Bala-Carter Levi of $\check \BO(m)^{(u)}$. From the description of the roots $\Phi(\check G^{(u)})$, there is a natural bijection between simple roots of $\check G^{(u)}$ and those of $G$. In particular, the Levi $\check L^{(u)}$ of $\check G^{(u)}$ corresponds to a Levi $L$ of $G$. Let $\bd^{(u)}_L$ be the covering duality map for $\check L^{(u)}$ and $L$. Let $\fz(\fl)$ denote the center of $\fl$ and let $P = LU$ be any parabolic subgroup containing $L$ as its Levi.

Our main conjecture is the following.

\begin{conjecture}\label{conj:main}
	In the above situation, we have
	\begin{equation*}
		X_{L_k(\fg)} = \overline{\cS(\fl, \bd^{(u)}_L \BO_{\check L^{(u)}})}.
	\end{equation*}
\end{conjecture}

If $u =1$, we recover Conjecture \ref{conj:int} for integer levels. 

\begin{namedtheorem}[Corollary of Conjecture,]~
	\begin{enumerate}
		\item We have $X_{L_k(\fg)} \cap \cN = \overline{\bd^{(u)} \check \BO(m)^{(u)}}$.
		
		\item $X_{L_k(\fg)} \subseteq \cN$ if and only if $\check \BO(m)^{(u)}$ is distinguished. In particular, the quasi-lisse-ness of $L_k(\fg)$ depends only on $m$, not on $u$.
	\end{enumerate}
\end{namedtheorem}

\begin{proof}
	By Lemma \ref{lem:sheets-cap-N-ss}(1), $X_{L_k} \cap \cN = \overline{\Ind_L^G \bd^{(u)}_L \BO_{\check L^{(u)}} }$. This is equal to $\overline{\bd^{(u)} \check \BO(m)^{(u)}}$ since $\bd^{(u)} \Sat_{\check L^{(u)}}^{\check G^{(u)}} = \Ind_L^G \bd_L^{(u)}$ by \cite[Theorem 1.1(ii)]{Gao-Liu-Lo-Shahidi}. This proves (1). For part (2), $X_{L_k} \subseteq \cN$ implies $X_{L_k} \cap \fg_{ss} = \{0\}$. Since $X_{L_k} \cap \fg_{ss} = \Ad G \cdot \fz(\fl)$ by Lemma \ref{lem:sheets-cap-N-ss}(2), we must have $\fl = \fg$. Conversely, if $\fl = \fg$, then from (\ref{eqn:sheet-closure}) we must have $X_{L_k} \subseteq \cN$. Finally, note that $\fl = \fg$ if and only if $\check \fl^{(u)} = \check \fg^{(u)}$, which happens if and only if $\check \BO(m)^{(u)}$ is distinguished in $\check \fg^{(u)}$. 
\end{proof} 

\begin{remark}
	If we identify $\check \fg^{(u)}$ with $\check \fg$ as in \ref{cls:du-ADE} so that $\bd^{(u)}$ is a map between nilpotent orbits of $\check \fg$ and $\fg$, then by Proposition \ref{prop:du-reg} below, we have
	\begin{equation*}
		X_{L_k(\fg)} \cap \cN= \overline{\bd^{(u)} \bd^{(m)} \BO_{reg}}
	\end{equation*}
	where $\BO_{reg}$ is the regular nilpotent orbit in $\fg$.
\end{remark}

%------
\subsection{Explicit known results and conjectures}\label{subsec:explicit-results}

When $u = 1$ (i.e. $k$ is an integer), a list of known cases of our conjecture is collected in \cite[Table 2]{SYZ:cl}. In what follows we mainly focus on cases where $k$ is not an integer. By abuse of notation we will write $\check \BO(m)^{(u)}$ as $\check \BO(m)$.

\begin{clause}[Admissible levels]\label{cls:adm}
	Let $\fg$ be of simply-laced type. The level $k$ is said to be \textit{admissible} if $m \ge \check \Bh$. Since the largest possible value of the cyclotomic level map $\cl$ is $\check \Bh$, the condition $m \ge \check \Bh$ implies $\check \BO(m) = \check \BO(\check \Bh) = \check \BO_{reg}$, the regular orbit in $\check \fg$.
	
	\begin{proposition}\label{prop:du-reg}
		We have $\bd^{(u)} \check \BO_{reg} = \BO(u)$.
	\end{proposition}
	
	This equality can be obtained from \cite[Theorem 1.1]{Arakawa-van-Ekeren-Moreau:Oq}. It was also observed by Fan Gao, see \cite[Remark 5.1]{Gao-Liu-Lo-Shahidi}. We record an alternative proof here.
	
	\begin{proof}
		Let $\xi_u \in \fh_{aff}^*$ be the unique dominant $W_{aff}$-translate of $u \Lambda_0 + \rho$ defined in \S \ref{subsec:cl} (for $u$ in place of $m$), and let $W_{\xi_u} \subset W_{aff}$ be the stabilizer of $\xi_u$, so that the element $w_u$ defined in \S \ref{subsec:cl} is the longest element of $W_{\xi_u}$. Then $\BO(u)$ corresponds, under Springer correspondence, to Lusztig-Spaltenstein's $j$-induction $j_{W_{\xi_u}}^W \sgn$ of the sign representation of $W_{\xi_u}$, see \cite[Lemma 4.2.3]{SYZ:cl} or \cite[Proposition 4.8]{Yun:Epi}. Here we are identifying the subgroup $W_{\xi_u} \subset W_{aff}$ with a subgroup of $W$ via the projection $\pi: W_{aff} = W \ltimes Q \surj W$, where $Q$ is the coroot lattice of $\fg$. On the other hand, by \ref{cls:du-ADE}, $\bd^{(u)} \check \BO_{reg}$ is the unique open orbit in $\AV(\cU(\fg)/J_{\rho/u})$, which by \cite[3.10 Theorem]{Joseph:prim} or \cite[Proposition 2.13]{Bai-Gao-Wang-Xie} corresponds to $j_{W_{[\rho/u]}}^W \sgn$, where $W_{[\rho/u]}$ is the integral Weyl group of $\rho/u$. It remains to show that $W_{\xi_u}$, or rather $\pi(W_{\xi_u})$, is conjugate to $W_{[\rho/u]}$ inside $W$.
		
		Let $\Phi_{u \Lambda_0 + \rho} \subseteq \Phi_{aff}$ be the set of affine roots $\tilde \alpha$ that are orthogonal to $u \Lambda_0 + \rho$, and let $\Phi_{[\rho/u]} \subseteq \Phi = \Phi_{fin}$ be the set of finite roots integral to $\rho/u$. Let $W_{u \Lambda_0+ \rho} = \operatorname{Stab}_{W_{aff}}(u \Lambda_0 + \rho)$. Then $W_{u \Lambda_0+ \rho}$ (resp. $W_{[\rho/u]}$) is generated by reflections of roots in $\Phi_{u \Lambda_0 + \rho}$ (resp. $\Phi_{[\rho/u]}$). Suppose $\tilde \alpha = \alpha + n \delta$ where $\alpha \in \Phi$. Then $\tilde \alpha$ is in $\Phi_{u \Lambda_0 + \rho}$ if and only if $0 = \langle \check \alpha + n K, u \Lambda_0 + \rho \rangle = nu + \langle \check \alpha, \rho \rangle = nu + \operatorname{ht}(\alpha)$, if and only if $\tilde \alpha = \alpha - \frac{\operatorname{ht}(\alpha)}u \delta$, where $\operatorname{ht}(\alpha)$ denotes the height of $\alpha$. In particular, $u$ divides $\operatorname{ht}(\alpha)$. Hence $\langle \check \alpha, \frac \rho u \rangle = \frac{\operatorname{ht}(\alpha)}u \in \BZ$ and $\alpha \in \Phi_{[\rho/u]}$. Conversely, each $\alpha \in \Phi_{[\rho/u]}$ determines uniquely the affine root $\tilde \alpha = \alpha - \frac{\operatorname{ht}(\alpha)}u \delta$ in $\Phi_{u \Lambda_0 + \rho}$. In other words, we have a bijection
		\begin{equation*}
			\Phi_{u \Lambda_0 + \rho} \bijects \Phi_{[\rho/u]},\quad
			\tilde \alpha = \alpha - \frac{\operatorname{ht}(\alpha)}u \delta \mapsto \alpha.
		\end{equation*}
		It induces an isomorphism 
		\begin{equation*}
			W_{u \Lambda_0 + \rho} \bijects W_{[\rho/u]},\quad
			s_{\tilde \alpha} = s_{\big( \alpha - \frac{\operatorname{ht}(\alpha)}u \delta \big)} \mapsto s_\alpha.
		\end{equation*}
		In fact this isomorphism is the restriction of the projection $\pi: W_{aff} \surj W$ above, since $s_{\big( \alpha - \frac{\operatorname{ht}(\alpha)}u \delta \big)} = s_\alpha t_{\big(- \frac{\operatorname{ht}(\alpha)}u \check \alpha\big)}$ where $t_\mu \in W_{aff}$ denotes the translation operator of $\mu \in Q$. Hence we have $\pi\big( W_{u \Lambda_0 + \rho} \big) = W_{[\rho/u]}$.
		
		Finally, $W_{u \Lambda_0 + \rho}$ is a conjugate of $W_{\xi_u}$ by some element $w t_\mu \in W \ltimes Q = W_{aff}$ because $u \Lambda_0 + \rho$ is a $W_{aff}$-translate of $\xi_u$. Since $\pi \circ \Ad(w t_\mu) = \Ad(w) \circ \pi$, we have 
		\begin{equation*}
			\pi\big( W_{\xi_u} \big) = \pi\big( \Ad(w t_\mu) W_{u \Lambda_0 + \rho} \big) = \Ad(w) \pi\big(W_{u \Lambda_0 + \rho} \big) = \Ad(w) W_{[\rho/u]},
		\end{equation*}
		as required.
	\end{proof}
	
	As a result, our conjecture says
	\begin{equation*}
		X_{L_k} = \overline{\BO(u)}.
	\end{equation*}
	This was proven in \cite[Theorem 5.7.1]{Arakawa:C2}.
	
	The appearance of the orbits $\BO(m)$ on both sides of the Langlands duality is a fascinating phenomenon that seems to be a shadow of something deeper.
\end{clause}

\begin{clause}[$A_2$]
	Let $\fg = \fsl_3$. We have three families of levels based on the value of $m$.
	\begin{itemize}
		\item $m \ge \check \Bh = 3$. In this case $k$ is admissible, and we have seen in \ref{cls:adm} that our conjecture on $X_{L_k}$ holds.
		
		\item $m = 2$. In this case $\check \BO(2)$ is the minimal/subregular orbit. If $u =1$, then our conjecture says $X_{L_k} =\overline{\cS(\fl, \{0\})}$ where $\fl \subset \fg$ is any Levi of semisimple type $A_1$. This was proven in \cite{Jiang-Song, Arakawa-Moreau:sheets, AFK}. If $u > 1$, our conjecture says
		\begin{equation*}
			X_{L_k} = \overline{\cS(\fl, \BO_{L,reg})}
		\end{equation*}
		where $\fl$ is again of semisimple type $A_1$ and $\BO_{L,reg}$ is the regular orbit in $\fl$. This was confirmed by \cite[Theorem 1.3]{Jiang-Song}. Indeed, \textit{loc. cit.} says $X_{L_k} = \overline{\Ad G \cdot (\BC^\times (h_1 - h_2) + f_\theta)}$ where 
		\begin{equation*}
			h_1 = 
			\left(
			\begin{smallmatrix}
				1\\
				&-1\\
				&&0
			\end{smallmatrix}
			\right),\quad
			h_2 = 
			\left( 
			\begin{smallmatrix}
				0\\
				&1\\
				&&-1
			\end{smallmatrix}
			\right) ,\quad
			f_\theta =
			\left( 
			\begin{smallmatrix}
				&&1\\
				&0\\
				0
			\end{smallmatrix}
			\right).
		\end{equation*}
		If $\fl$ is chosen to contain the diagonal Cartan and has the highest root as its simple root, then $h_1 - h_2 \in \fz(\fl)^{reg}$, $f_\theta \in \BO_{L,reg}$, and hence $\overline{\Ad G \cdot (\BC^\times (h_1 - h_2) + f_\theta)} = \overline{J_G(\fl, \BO_{L,reg})} = \overline{\cS(\fl, \BO_{L,reg})}$.
		
		\item $m=1$. In this case $\check \BO(1)= 0$, and our conjecture says $X_{L_k} = \overline{\cS(\fh,0)} = \fg$. Indeed, by \cite[0.2.1 Theorem]{Gorelik-Kac}, $V^k$ is irreducible, and hence $L_k = V^k$, $X_{L_k} = X_{V^k} = \fg$.
	\end{itemize}
\end{clause}

\begin{clause}[$A_3$, $m/u = 3/2$]
	Let $\fg = \fsl_4$ and $m/u = 3/2$ (so that $k = -5/2$). In this case $\check \BO(3)$ is given by partition $[3,1]$ and is the saturation of the regular orbit $\BO_{\check L, reg}$ in any Levi $\check \fl$ of semisimple type $A_2$. We have $\bd_L^{(2)} \BO_{\check L, reg} = \BO_{L, min}$, the minimal/subregular orbit in $\fl$. Hence our conjecture says
	\begin{equation*}
		X_{L_k} = \overline{\cS(\fl, \BO_{L, min})}.
	\end{equation*}
	This agrees with \cite[Theorem 5.1]{Fasquel:thesis}. Indeed, it is proven in \textit{loc. cit.} that $X_{L_k} = \overline{J_G(x_0)}$ where 
	\begin{equation*}
		x_0 = \left(\begin{smallmatrix}
			3\\
			&-1 & & 1\\
			&& -1\\
			&&&-1
		\end{smallmatrix}\right)
		= \left(\begin{smallmatrix}
			3\\
			&-1\\
			&&-1\\
			&&&-1
		\end{smallmatrix}\right)
		+ \left(\begin{smallmatrix}
			0\\
			&0&&1\\
			&&0\\
			&&&0
		\end{smallmatrix}
		\vphantom{\begin{smallmatrix}
				3\\
				&-1\\
				&&-1\\
				&&&-1
		\end{smallmatrix}}\right)
		=: h+e
	\end{equation*}
	where $h$ (resp. $e$) is the semisimple part (resp. nilpotent part) of $x_0$. Then $\fl = \fz_\fg(h)$ is a Levi of semisimple type $A_2$, $h$ is in $\fz(\fl)^{reg}$, and $e$ is in the minimal nilpotent orbit in $\fl$. As a result $\overline{J_G(x_0)} = \overline{J_G(\fl, \BO_{L,min})} = \overline{\cS(\fl, \BO_{L,min})}$.
\end{clause}

\begin{clause}[$D_4, m = 4$]
	Let $\fg = \fso_8$ and $m = 4$. Then $\check \BO(4) = \check \BO_{sreg}$ is the subregular orbit, which is distinguished in $\check \fg^{(u)}$.
	\begin{itemize}
		\item If $u = 1$, our conjecture says $X_{L_k} = \overline{\bd \check \BO_{sreg}} = \overline{\BO_{min}}$ is the closure of the minimal orbit. This was confirmed by \cite[Theorem 1.1]{Arakawa-Moreau:Omin}. 
		
		\item If $u=3$, then our conjecture says $X_{L_k} = \overline{\bd^{(3)} \check \BO_{sreg}} = \overline{\BO_{sreg}}$. This was proven in \cite[Theorem 1.3]{Adamovic-Vukorepa}.
		
		\item If $u \ge 5$ is odd, our conjecture says $X_{L_k} = \overline{\bd^{(u)} \check \BO_{sreg}} = \overline{\BO_{reg}} = \cN$. This aligns with \cite[Conjecture 1.1, Theorem 1.2]{Adamovic-Vukorepa}.
	\end{itemize}
\end{clause}

\begin{clause}[$E_6, m/u = 9/7$]
	In this case $\check \BO(9) = \check \BO_{E_6(a_1)} = \check \BO_{sreg}$ and our conjecture says $X_{L_k} = \overline{\bd^{(7)} \check \BO_{sreg}} = \overline{\check \BO_{sreg}}$. This agrees with \cite[\S5.2, Conjecture 3]{AACLMMP}.
	
	More evidence for $m = 9$ is contained in \ref{cls:collapse-E} below.
\end{clause}

%-------------
\subsection{Relation to $\cW$-algebras}

In the next subsection, we collect some computational results and conjectures on $\cW$-algebras as evidence to our conjecture. In this subsection we recall relevant notions and explain the connection between these results and our main conjecture.

\begin{clause}[Affine $\cW$-algebras and Slodowy slices]
	Let $f \in \cN$ and write $\BO_f$ for the orbit of $f$. Write
	\begin{equation*}
		\BS_f = f + \fz_\fg(e)
	\end{equation*}
	for the Slodowy slice at $f$, where $\{h,e,f\}$ is an $\fsl_2$-triple containing $f$ and $\fz_\fg(e)$ is the centralizer of $e$ in $\fg$.
	
	The \textit{universal affine $\cW$-algebra} attached to $\fg$, $f$ and $k$ is the vertex algebra
	\begin{equation*}
		\cW^k(\fg,f) = H_f^0(V^k(\fg))
	\end{equation*}
	obtained as the result of quantized Drinfeld-Sokolov reduction applied to $V^k(\fg)$, see \cite{Feigin-Frenkel}, \cite{Kac-Roan-Wakimoto}, or \cite[\S 4]{Arakawa:C2}. Its associated variety is equal to the Slodowy slice $\BS_f$. The unique simple quotient of $\cW^k(\fg,f)$ will be denoted by $\cW_k(\fg,f)$.
	
	The following facts are proven in \cite{Arakawa:C2}.
	\begin{enumerate}
		\item $H_f^0(L_k) \neq 0$ if and only if $\BO_f \subseteq X_{L_k}$.
		\item If $\BO_f \subseteq X_{L_k}$, we have $X_{H_f^0(L_k)} = X_{L_k} \cap \BS_f$.
		\item The quotient map $V^k \surj L_k$ induces a quotient map $\cW^k(\fg,f) = H_f^0(V^k) \surj H_f^0(L_k)$. In particular, if $\BO_f \subseteq X_{L_k}$, then $H_f^0(L_k)$ admits a quotient map onto $\cW_k(\fg,f)$.
	\end{enumerate}
	It is conjectured by Kac-Wakimoto \cite{Kac-Wakimoto:rat-W} that if $\BO_f \subseteq X_{L_k}$, $H_f^0(L_k)$ is isomorphic to $\cW_k(\fg,f)$. This would imply
	\begin{equation*}
		X_{\cW_k(\fg,f)} \xeq{conj.} X_{H_f^0(L_k(\fg))} = X_{L_k(\fg)} \cap \BS_f
	\end{equation*}
	and hence
	\begin{equation*}%\label{eqn:W-dim}
		\dim X_{\cW_k(\fg,f)} \xeq{conj.} \dim X_{L_k(\fg)} \cap \BS_f.
	\end{equation*}
	If Conjecture \ref{conj:main} holds, then by (\ref{eqn:dim-of-sheet}) and \cite[Corollary 1.3.8(iii)]{Ginzburg:HC-Slodowy} the right side can be rewritten as
	\begin{equation*}
		\dim X_{L_k(\fg)} \cap \BS_f = \dim \overline{\cS(\fl, \bd^{(u)}_L \BO_{\check L^{(u)}})} - \dim \BO_f = \dim \fz(\fl) + \dim \bd^{(u)} \check \BO(m) - \dim \BO_f.
	\end{equation*}
\end{clause}

\begin{clause}[Collapsing levels]
	Let $\cW_k(\fg,f)$ be the simple $\cW$-algebra as defined above. The level $k$ is said to be \textit{collapsing} for $\cW_k(\fg,f)$ if 
	\begin{equation*}
		\cW_k(\fg,f) \cong L_{k^\natural}(\fg^\natural)
	\end{equation*}
	where $\fg^\natural$ is the centralizer of an $\fsl_2$-triple containing $f$, and $k^\natural$ is determined by the condition that $V^{k^\natural}(\fg^\natural)$ is a vertex subalgebra of $\cW^k(\fg,f)$. If moreover $\fg^\natural = \fg_1^\natural \oplus \cdots \oplus \fg_s^\natural$ where each $\fg_i^\natural$ is either abelian or simple, then
	\begin{equation*}
		L_{k^\natural}(\fg^\natural) = L_{k_1^\natural}(\fg_1^\natural) \otimes \cdots \otimes L_{k_s^\natural}(\fg_s^\natural)
	\end{equation*}
	for some levels $k_i^\natural$. Further assuming $\BO_f \subseteq X_{L_k(\fg)}$ and Kac-Wakimoto's conjecture that $H_f^0(L_k(\fg)) \cong \cW_k(\fg,f)$, we obtain isomorphisms of Poisson varieties
	\begin{subequations}
		\begin{equation}\label{eqn:collapse-AV}
			X_{L_k(\fg)} \cap \BS_f = X_{\cW_k(\fg,f)} \cong X_{L_{k_1^\natural}(\fg_1^\natural)} \times \cdots \times X_{L_{k_s^\natural}(\fg_s^\natural)}
		\end{equation}
		and
		\begin{equation}\label{eqn:collapse-AVN}
			X_{L_k(\fg)} \cap \cN \cap \BS_f \cong \big( X_{L_{k_1^\natural}(\fg_1^\natural)} \cap \cN_1^\natural \big) \times \cdots \times \big( X_{L_{k_s^\natural}(\fg_s^\natural)} \cap \cN_s^\natural \big)
		\end{equation}
		where $\cN_i^\natural$ is the nilpotent cone of $\fg_i^\natural$. 
	\end{subequations}	
	When the associated varieties on both sides of (\ref{eqn:collapse-AV}-\ref{eqn:collapse-AVN}) are either known or are covered by our conjecture, the equalities on dimensions between the two sides provide a partial evidence for our conjecture. 
\end{clause}

\begin{clause}[Singularities in nilpotent orbit slices]\label{cls:sing}
	According to our conjecture, $X_{L_k} \cap \cN = \overline{\BO}$ for some nilpotent orbit $\BO$, and the left side of (\ref{eqn:collapse-AVN}) is the intersection $\overline{\BO} \cap \BS_f$. Some such intersections have been computed in \cite{Kraft-Procesi:sing-A,Kraft-Procesi:sing-BCD} in classical types and in \cite{Fu-Juteau-Levy-Sommers} in exceptional types. In those cases it is possible to verify the equality (\ref{eqn:collapse-AVN}) in a stronger sense.
	
	In Claim \ref{clm:sing-A} and \ref{clm:sing-D} we match the singularity types on both sides of (\ref{eqn:collapse-AVN}) in certain cases. Recall that two pointed varieties $(X,x)$, $(Y,y)$ are said to be \textit{smoothly equivalent} if there exists another pointed variety $(Z,z)$ together with morphisms
	\begin{equation*}
		(X,x) \xleftarrow{\varphi} (Z,z) \xrightarrow{\psi} (Y,y)
	\end{equation*}
	which are smooth at $z$. This defines an equivalence relation on pointed varieties, and we denote this by $\operatorname{Sing}(X,x) = \operatorname{Sing}(Y,y)$, see \cite[\S2.1 Definition]{Kraft-Procesi:sing-A}. 
	
	In \ref{cls:collapse-E} and Tables \ref{tbl:E6}-\ref{tbl:E8} below we include descriptions of the pointed spaces $\big( X_{L_k(\fg)} \cap \cN \cap \BS_f,f  \big) = \big( \overline{\BO} \cap \BS_f, f \big)$ up to algebraic isomorphisms (not just the singularity type) whenever they were computed in \cite{Fu-Juteau-Levy-Sommers}. There are four types of descriptions with the following meanings: we say a pointed space $(X,x)$ is a
	\begin{itemize}
		\item Singleton, denoted by ``$\ast$'', if $X = \{x\}$;
		
		\item Kleinian/simple surface singularity ``$\cK$'', where $\cK\in \{A_n,B_n,C_n,D_n,E_6,E_7,E_8,F_4,G_2\}$, if $( X, x ) \cong ( \cN' \cap \BS_{f'}, f')$ where $\cN'$ is the nilpotent cone of a simple Lie algebra $\fg'$ of type $\cK$, and $f'$ is a subregular nilpotent element of $\fg'$;
		
		\item Minimal singularity ``$\cM$'', where $\cM \in \{a_n,b_n,c_n,d_n,e_6,e_7,e_8,f_4,g_2\}$, if $(X,x) \cong (\overline{\BO_{min}'},0)$ where $\BO_{min}'$ is the minimal nilpotent orbit in a simple Lie algebra $\fg'$ of type $\cM$, and $0$ is the zero element in $\fg'$;
		
		\item Full nilpotent cone ``$\cN_{\fg'}$'' if $(X,x) \cong (\cN_{\fg'},0)$ where $\fg'$ is a semisimple Lie algebra.
	\end{itemize}
\end{clause}

%--------
\subsection{Evidence from $\cW$-algebras}

\begin{clause}[Evidence from \cite{Arakawa-Creutzig-Kawasetsu}]
	Consider the following cases, taken from \cite[Table 1]{Arakawa-Creutzig-Kawasetsu}:
	\begin{center}
		\begin{tabular}{c|cccc}
			$\fg$ & $\fso_8$ & $\fe_6$ & $\fe_7$ & $\fe_8$ \\ \hline
			$m/u$ & $5/4$ & $10/9$ & $15/14$ & $25/24$\\
			& & $11/10$ & $16/15$ & $26/25$\\
			&&& $17/16$ & $27/26$\\
			&&&& $28/27$\\
			&&&& $29/28$
		\end{tabular}
	\end{center}
	In all these cases, $\check \BO(m) = \check \BO_{sreg}$ is the subregular orbit which is distinguished, and $\bd^{(u)} \check \BO_{sreg} = \BO_{reg}$ is the regular orbit. Hence our conjecture says
	\begin{equation*}
		X_{L_k} = \overline{ \bd^{(u)} \check \BO_{sreg} } = \overline{\BO_{reg}} = \cN.
	\end{equation*}
	Consequently, for $f_{reg} \in \BO_{reg}$, we should have
	\begin{equation*}
		X_{H_{f_{reg}}^0(L_k)} = X_{L_k} \cap \BS_{f_{reg}} = \{f\}
	\end{equation*}
	and $H_{f_{reg}}^0(L_k)$ is lisse. In particular, as a quotient of $H_{f_{reg}}^0(L_k)$, $\cW_k(\fg,f_{reg})$ is lisse. 
	
	The lisse-ness of these $\cW_k(\fg,f_{reg})$ was conjectured in \cite[Conjecture 3]{Arakawa-Creutzig-Kawasetsu} and, for the levels on the first row, was proven in Theorem 1.2 of \textit{op. cit.}
\end{clause}

\begin{clause}[Evidence from \cite{AACLMMP}]
%	The work \cite{AACLMMP} contains a large list of collapsing levels. In the remainder of this subsection, we collect from \textit{op. cit.} the cases where $\fg$ is simply-laced and $L_k(\fg)$ is non-admissible. Whenever $f \in X_{L_k(\fg)}$, we verify the equalities on the dimensions of associated varieties resulting from (\ref{eqn:collapse-AV}-\ref{eqn:collapse-AVN}) assuming our conjecture on $X_{L_k(\fg)}$. We also list the singularities of $X_{L_k(\fg)} \cap \cN \cap \BS_f$ and verify the equality (\ref{eqn:collapse-AVN}) or $\operatorname{Sing}$(\ref{eqn:collapse-AVN}) whenever they are computed (see \ref{cls:sing} for notations).

	The work \cite{AACLMMP} contains a large list of collapsing levels. In the remainder of this subsection, we collect from \textit{op. cit.} the cases where $\fg$ is simply-laced and $L_k(\fg)$ is non-admissible. Whenever $f \in X_{L_k(\fg)}$, we list the singularities of $X_{L_k(\fg)} \cap \cN \cap \BS_f$ and verify the equality (\ref{eqn:collapse-AVN}) or $\operatorname{Sing}$(\ref{eqn:collapse-AVN}) whenever they are computed (see \ref{cls:sing} for notations). We then verify the equalities on the dimensions of associated varieties resulting from (\ref{eqn:collapse-AV}-\ref{eqn:collapse-AVN}) assuming our conjecture on $X_{L_k(\fg)}$. 
\end{clause}

\begin{clause}[Type $A$]\label{cls:collapse-A}
	Let $\fg = \fsl_n$. Let $[q^l,1^r]$ be a partition of $n$ such that $q \ge 2$, $l \ge 1$, and let $f \in \BO_{[q^l,1^r]}$. Let $k = -n + \frac{n-l}{q-1}$. Then by \textit{op. cit.} \S 4.1 we have
	\begin{equation*}
		\cW_{-n+\frac{n-l}{q-1}}(\fsl_n,[q^l,1^r]) = \cW_k(\fg,f) 
		\cong L_{k^\natural}(\fg^\natural) := M(1) \otimes L_{-l+\frac{n-l}{q-1}}(\fsl_l) \otimes L_{-r+\frac{r}{q-1}}(\fsl_r).
	\end{equation*}
	Here $M(1)$ is the rank $1$ Heisenberg vertex algebra. For simplicity, we only consider the cases where the fractions appearing in the levels are in lowest terms, that is, $\gcd(n-l,q-1) = \gcd(r,q-1) =1$. We have $n-l \ge l$ (this is equivalent to $(q-2)l + r \ge 0$, which is true by the assumption that $q \ge 2$), so together with the condition $\gcd(n-l,q-1) = \gcd(r,q-1) =1$, the last two factors on the right hand side are admissible. 
	
	Write $l = a(q-1)+b$, $r = c(q-1)+d$ with $b,d <q-1$. Write $\BO_c(m)$ for the orbit $\BO(m)$ inside $\fsl_c$. The unique open orbit in $X_{L_k(\fg)}$ is, according to \cite[\S 3]{Gao-Liu-Lo-Shahidi},
	\begin{equation*}
		\bd^{(q-1)}\check \BO_n(n-l) = \bd^{(q-1)} \check \BO_{[n-l,l]} = \BO_{[(q-1)^{l+c},d]+[(q-1)^a,b]}
		= \BO_{[(2q-2)^a,q-1+b,(q-1)^{l+c-a-1},d]}.
	\end{equation*}
	This visibly contains $f$ in its closure. Similarly
	\begin{equation*}
		X_{L_{-l+\frac{n-l}{q-1}}(\fsl_l)} \cap \cN(\fsl_l)
		= \overline{\bd^{(q-1)} \check \BO_l(n-l)} = \overline{\BO_l(u)} = \overline{\BO_{[(q-1)^a,b]}}
	\end{equation*}
	and
	\begin{equation*}
		X_{L_{-r+\frac{r}{q-1}}(\fsl_r)} \cap \cN(\fsl_r) = \overline{\bd^{(q-1)} \check \BO_r(r)} = \overline{\BO_r(q-1)} = \overline{\BO_{[(q-1)^c,d]}}.
	\end{equation*}
	
	\begin{claim}\label{clm:sing-A}
		Suppose $l = 1$ (i.e. $a=0$ and $b=l$). Assume Conjecture \ref{conj:main} holds. Then $f \in X_{L_k(\fg)}$ and $\operatorname{Sing}$(\ref{eqn:collapse-AVN}) holds, that is
		\begin{equation*}
			\operatorname{Sing}\big( X_{L_k(\fg)} \cap \cN \cap \BS_f, f \big)
			= \operatorname{Sing}\big( X_{L_{k^\natural}(\fg^\natural)} \cap \cN^\natural, 0 \big).
		\end{equation*}
	\end{claim}
	
	In this case, we have
	\begin{align*}
		X_{L_k(\fg)} \cap \cN \cap \BS_f &= \overline{\BO_{[q,(q-1)^c,d]}} \cap \BS_{[q,1^r]}\\
		X_{L_{k^\natural}(\fg^\natural)} \cap \cN^\natural &= 0 \times \overline{\BO_{[(q-1)^c,d]}}
	\end{align*}
	(note that the $\fsl_l = \fsl_1$ factor is gone). By \cite[\S3.1 Proposition, \S4.1 Lemma]{Kraft-Procesi:sing-A}, we have
	\begin{equation*}
		\operatorname{Sing}\big( \overline{\BO_{[q,(q-1)^c,d]}} \cap \BS_{[q,1^r]}, f \big)
		= \operatorname{Sing}\big( \overline{\BO_{[q,(q-1)^c,d]}}, \BO_{[q,1^r]} \big)
		= \operatorname{Sing}\big( \overline{\BO_{[(q-1)^c,d]}}, 0 \big).
	\end{equation*}
	This proves the claim. \qed
	
	\begin{claim}\label{clm:collapse-A}
		Assume Conjecture \ref{conj:main} holds. Then in the setup of \ref{cls:collapse-A} (without assuming $l=1$),
		\begin{enumerate}
			\item $f \in X_{L_k(\fg)}$;
			\item $\dim$(\ref{eqn:collapse-AV}) holds, that is, $\dim X_{L_k(\fg)} \cap \BS_f = \dim X_{L_{k^\natural}(\fg^\natural)}$;
			\item $\dim$(\ref{eqn:collapse-AVN}) holds, that is, $\dim X_{L_k(\fg)} \cap \cN \cap \BS_f = \dim X_{L_{k^\natural}(\fg^\natural)} \cap \cN^\natural$.
		\end{enumerate}
	\end{claim}

%	Since $n-l$ and $q-1$ may be not coprime, we let
%	\begin{itemize}
%		\item $d = \gcd(n-l,q-1)$,
%		\item $m = \frac{n-l}d$,
%		\item $u = \frac{q-1}d$.
%	\end{itemize}
%	Then $\frac mu = \frac{n-l}{q-1} = k+ n$ and $\gcd(m,u) = 1$. Since $n = ql+r$, we have $n-l = (q-1)l+r$, and hence $d = \gcd(r,q-1)$. 
%	
%	\begin{lemma}
%		We have $n = dm+l$ and $l < m$.
%	\end{lemma}
%	
%	\begin{proof}
%		Write
%		\begin{itemize}
%			\item $n = xm + t$ where $x$ is the largest integer so that $xm \le n$,
%			\item $l = ym+v$ where $y$ is the largest integer so that $ym \le l$.
%		\end{itemize}
%		Then $md = n-l = (x-y)m + t-v$, and $(x-y-d)m = v-t$. By definition both $v$ and $t$ are smaller than $m$, and so we must have $x-y-d = v-t=0$. As a result 
%		\begin{align*}
%			r &= n - l - (q-1)l = md - du(ym+t)\\
%			&= (1-yu)dm - tdu 
%		\end{align*}
%		and $r \ge 0$ forces either
%		\begin{itemize}
%			\item $y=u=1$ and $t=0$, or
%			\item $y=0$.
%		\end{itemize}
%		In the first case, the equation $n = ql + r$ becomes $(d+1)m = 1 \cdot m + d m$ and $q = 1$ which violates the assumption that $q \ge 2$. In the second case, we have $x = y+d = d$ and $t = v = l < m$, as required.%Finally, $r \ge 0$ becomes $d(m-lu) \ge 0$ which implies $lu \le m$.
%	\end{proof}

	The equality $\dim X_{L_k(\fg)} \cap \cN \cap \BS_f = \dim X_{L_{k^\natural}(\fg^\natural)} \cap \cN^\natural$, which is equivalent to $\dim \bd^{(q-1)}\check \BO_n(n-l) - \dim \BO_{[q^l,1^r]} = \dim \bd^{(q-1)} \check \BO_l(n-l) + \dim \bd^{(q-1)} \check \BO_r(r)$, can be verified using the dimension formula for nilpotent orbits, see for example \cite[Corollary 6.1.4]{Collingwood-McGovern}. By the dimension formula for sheets (\ref{eqn:dim-of-sheet}), part (2) of Claim \ref{clm:collapse-A} amounts to saying that the center of the Bala-Carter Levi of $\check \BO_n(n-l)$ on the left side is equal to $1 = \dim X_{M(1)}$. This is true, as the Bala-Carter Levi of $\check \BO_n(n-l)$ is $\fs(\fgl_{n-l} \times \fgl_l)$, where $\fs(-)$ takes trace zero elements. %This completes the verification of Claim \ref{clm:collapse-A}. 
	\qed
	
	%We are also able to verify that the singularity types of the two sides of (\ref{eqn:collapse-AVN}) match, whenever they were computed in \cite{Kraft-Procesi:sing-A}. We omit the details.

\end{clause}

\begin{clause}[Type $D$]\label{cls:collapse-D}
	Let $\fg = \fso_{2n}$. Let $[q^l,1^r]$ be a partition of $2n$ of type $D$ such that $q \ge 2$, $l \ge 1$, and let $f \in \BO_{[q^l,1^r]}$. By \cite[\S 4.2 and Table 1]{AACLMMP}, we have
	\begin{subequations}
	\begin{enumerate}[label=(\roman*)]
		\item If $q$ is odd,
		\begin{equation}\label{eqn:collapse-D-qodd}
			\cW_{-(2n-2) + \frac{2n-l}{q-1}}(\fso_{2n}, [q^l, 1^r]) =: \cW_k(\fg,f) \cong L_{k^\natural}(\fg^\natural) := L_{-(l-2) + \frac{2n-l}{q-1}} (\fso_l) \otimes L_{-(r-2)+ \frac r{q-1}}(\fso_r);
		\end{equation}
		\item If $q$ is even, $r>2$,
		\begin{equation}\label{eqn:collapse-D-qeven}
			\cW_{-(2n-2) + \frac{2n-2-l}{q-1}}(\fso_{2n}, [q^l, 1^r]) =: \cW_k(\fg,f) \cong L_{k^\natural}(\fg^\natural) := L_{-(\frac l2+1) + \frac{n-1-l/2}{q-1}}(\fsp_l) \otimes L_{-(r-2) + \frac{r-2}{q-1}}(\fso_r).
		\end{equation}
	\end{enumerate}
	\end{subequations}
	For simplicity, we only consider the cases where the fractions appearing in the levels are all in lowest terms. Since we do not yet have a general conjecture for $X_{L_{k^\natural}(\fg^\natural)}$ in type $B$ and $C$, we restrict ourselves to the cases where the right hand sides are admissible. 
	
	Suppose $q$ is odd. By our assumption, $\gcd(2n-l,q-1) = \gcd(r,q-1) = 1$. In particular, $l$ and $r$ are odd numbers. In this case $\check \BO(2n-l)$ is distinguished. To see this, notice that $2n-l > n$ (indeed, this is equivalent to $n-l = (\frac12 q - 1)l + \frac12 r > 0$ which is true since $q \ge 3$, $l \ge 1$, $r \ge 0$), and hence $\check \BO(2n-l)$ corresponds to the partition $[2n-l,l]$ with $2n-l>l$, indeed distinguished. As a result, $L_k(\fg)$ and hence both sides of (\ref{eqn:collapse-D-qodd}) are quasi-lisse. 
	
	We now write down the partition of the orbit dense in $X_{L_k(\fg)}$ and $X_{L_{k^\natural}(\fg^\natural)}$ using the formula for $\bd^{(u)}$ in \cite[\S 3]{Gao-Liu-Lo-Shahidi} (with a slight distinction, see Remark \ref{rmk:Spin-vs-SO}) and the formula for admissible representations in \cite{Arakawa:C2}. Write $l = a(q-1)+b$ where $a$ is the largest integer so that $a(q-1) \le l$. Similarly write $r = c(q-1)+d$.  
	\begin{center}
		\resizebox{\textwidth}{!}{
			\begin{tabular}{c|lll}
				Parity for $(a,c)$ & $X_{L_{-(2n-2) + \frac{2n-l}{q-1}}(\fso_{2n})}$ & $X_{L_{-(l-2) + \frac{2n-l}{q-1}} (\fso_l)}$ & $X_{ L_{-(r-2)+ \frac r{q-1}}(\fso_r)}$
				\\ \hline
				(even,even) & $[ (2q-2)^a,q-1+b,(q-1)^{c+l-a-1},d ]$ & $[(q-1)^a,b]$ & $[(q-1)^c,d]$\\
				(even,odd) & $[(2q-2)^a,q-1+b,(q-1)^{c+l-a-2},q-2,d,1]$ & $[(q-1)^a,b]$ & $[(q-1)^{c-1},q-2,d,1]$\\
				(odd,even) & $[(2q-2)^{a-1},2q-3,q-1+b,q,(q-1)^{c+l-a-2},d]$ & $[(q-1)^{a-1},q-2,b,1]$ & $[(q-1)^c,d]$\\
				(odd,odd) & $[(2q-2)^{a-1}, 2q-3, q-1+b,q,(q-1)^{c+l-a-3},q-2,d,1]$ & $[(q-1)^{a-1},q-2,b,1]$ & $[(q-1)^{c-1},q-2,d,1]$
			\end{tabular}
		}
	\end{center}
	
	Suppose now $q$ is even. Then both $l$ and $r$ are even. The right hand side $L_{k^\natural}(\fg^\natural)$ is admissible precisely when $n-1-l/2 \ge l/2+1$, i.e. when $n-2 \ge l$. In this case $\check \BO(2n-2-l)$ is distinguished, since $n-2 \ge l$ implies $2n-2-l \ge n$, and $\check \BO(2n-2-l)$ is given by the partition $[2n-1-l,l+1]$. As a result, both sides of (\ref{eqn:collapse-D-qeven}) are quasi-lisse. To write down the precise partitions of the orbits, we again let $l = a(q-1)+b$ and $r = c(q-1)+d$ with $b,d < q-1$. Then $a$ and $b$ (resp. $c$ and $d$) must have the same parity.
	
	\begin{center}
		\resizebox{\textwidth}{!}{
			\begin{tabular}{c|lll}
				Conditions for $(a,c,d)$ & $X_{L_{-(2n-2) + \frac{2n-2-l}{q-1}}(\fso_{2n})}$ & $X_{L_{-(\frac l2+1) + \frac{n-1-l/2}{q-1}}(\fsp_l)}$ & $X_{L_{-(r-2) + \frac{r-2}{q-1}}(\fso_r)}$
				\\ \hline
				(even,even,>0) & $[(2q-2)^a, q+b-1, (q-1)^{l+c-a-1}, d-1,1]$ & $[(q-1)^a,b]$ & $[(q-1)^c,d-1,1]$\\
				(even,even,=0) & $[(2q-2)^a, q+b-1, (q-1)^{l+c-a-1}]$ & $[(q-1)^a,b]$ & $[(q-1)^c]$\\
				(even,odd,any) & $[(2q-2)^a, q+b-1, (q-1)^{l+c-a-1}, d]$ & $[(q-1)^a,b]$ & $[(q-1)^c,d]$\\
				(odd,even,>0)  & $[(2q-2)^{a-1}, 2q-3, q+b, (q-1)^{l+c-a-1}, d-1, 1]$ & $[(q-1)^{a-1},q-2,b+1]$ & $[(q-1)^c,d-1,1]$\\
				(odd,even,=0)  & $[(2q-2)^{a-1}, 2q-3, q+b, (q-1)^{l+c-a-1}]$ & $[(q-1)^{a-1},q-2,b+1]$ & $[(q-1)^c]$\\
				(odd,odd,any)  & $[(2q-2)^{a-1}, 2q-3, q+b, (q-1)^{l+c-a-1}, d]$ & $[(q-1)^{a-1},q-2,b+1]$ & $[(q-1)^c,d]$
			\end{tabular}
		}
	\end{center}
	In all cases, we have $f \in X_{L_k(\fg)}$.
	
	\begin{claim}\label{clm:sing-D}
		Suppose Conjecture \ref{conj:main} holds. In the cases listed in the following table, $f \in X_{L_k(\fg)}$ and $\operatorname{Sing}$(\ref{eqn:collapse-AVN}) holds, that is 
		\begin{equation*}
			\operatorname{Sing}\big( X_{L_k(\fg)} \cap \cN \cap \BS_f, f \big)
			= \operatorname{Sing}\big( X_{L_{k^\natural}(\fg^\natural)} \cap \cN^\natural, 0 \big).
		\end{equation*}
		\begin{center}
			\begin{tabular}{c|ll|l}
				Conditions for $(q,a,b,c,d)$
				& $X_{L_k(\fg)} \cap \cN$ & $\BO_f$ & $X_{L_{k^\natural}(\fg^\natural)} \cap \cN^\natural$
				\\ \hline
				$(\text{odd},0,1,\text{even},\text{any})$ & $[q,(q-1)^c,d]$ & $[q,1^r]$ & $[1] \times [(q-1)^c,d]$
				\\ 
				$(\text{odd},0,1,\text{odd},\text{any})$ & $[q,(q-1)^{c-1},q-2,d,1]$ & $[q,1^r]$ & $[1] \times [(q-1)^{c-1},q-2,d,1]$
				\\
				$(3,1,1,\text{even},1)$ & $[3^3,2^{c-1},1]$ & $[3^3,1^{2c+1}]$ & $[1^3] \times [2^{c-1},1]$
				\\
				$(3,1,1,\text{odd},1)$ & $[3^3,2^{c-1},1^3]$ & $[3^3,1^{2c+1}]$ & $[1^3] \times [2^{c-1},1^3]$
			\end{tabular}
		\end{center}
	\end{claim}
	
	With the explicit partitions listed here, the claim follows from \cite[\S12.3 Theorem]{Kraft-Procesi:sing-BCD}. \qed

	\begin{claim}
		Suppose Conjecture \ref{conj:main} holds. Then in the setup of \ref{cls:collapse-D} (without  the conditions in Claim \ref{clm:sing-D}), we have
		\begin{enumerate}
			\item $f \in X_{L_k(\fg)}$;
			\item $L_k(\fg)$, $\cW_k(\fg,f)$, and $L_{k^\natural}(\fg^\natural)$ are quasi-lisse;
			\item $\dim$(\ref{eqn:collapse-AV}) and $\dim$(\ref{eqn:collapse-AVN}) hold, that is, $\dim X_{L_k(\fg)} \cap \BS_f = \dim X_{L_{k^\natural}(\fg^\natural)}$.
		\end{enumerate}
	\end{claim}
	
	This is again a straightforward verification using the dimension formula of nilpotent orbits. \qed	
\end{clause}

\begin{clause}[Type $E$]\label{cls:collapse-E}
	In Table \ref{tbl:E6}-\ref{tbl:E8} we collect collapsing non-admissible levels in type $E$. The data are either taken from \cite[Appendix A]{AACLMMP}, \cite{Arakawa:C2}, \cite[Table 6]{AFK}, or are computed using the \textsf{atlas} software. The table headings have the following meanings:
	\begin{itemize}
		\item The first three columns are the Bala-Carter label for the orbit $\BO_f$, the dimension of $\BO_f$, and the number $k + \check \Bh = \frac mu$.
		
		\item The fourth column contains the Bala-Carter label of the conjectural unique nilpotent orbit inside $X_{L_k(\fg)} \cap \cN$.
		
		\item The fifth column contains a pair of numbers $(d,r)$ where $d = \dim X_{L_k(\fg)} \cap \cN$ and $r = \dim X_{L_k(\fg)} \cap \fg_{ss}$ are computed based on our conjecture. We have $\dim X_{L_k(\fg)} = d+r$ by (\ref{eqn:dim-of-sheet}). If $r = 0$, then we write the number $d$ only.
		
		\item The sixth column contains the decomposition $L_{k^\natural}(\fg^\natural)$ into a tensor product of simple vertex algebras.
		
		\item The seventh and eighth columns are similar to the fourth and the fifth, but for $L_{k^\natural}(\fg^\natural)$ instead of $L_k(\fg)$. 
		
		\item The last column contains the descriptions of $X_{L_k(\fg)} \cap \cN$, whenever they are computed in \cite{Fu-Juteau-Levy-Sommers}.% (see \ref{cls:sing} for notations).
	\end{itemize}
	The equalities $\dim$(\ref{eqn:collapse-AV}) and $\dim$(\ref{eqn:collapse-AVN}) are equivalent to
	\begin{equation*}
		d - \dim \BO_f = d^\natural \text{ and } r = r^\natural.
	\end{equation*}
	This is true in all cases listed. Moreover, (\ref{eqn:collapse-AVN}) also holds whenever the singularities are known.
\end{clause}

\begin{table}[p]
	\caption{Collapsing non-admissible levels in $E_6$}\label{tbl:E6}
	\resizebox{\textwidth}{!}{
	%{\small
	\begin{tabular}{c|c|c||c|c||c|c|c||c}
	$f$ & $\dim \BO_f$ & $k+\check \Bh$ & $X_{L_k(\fg)} \cap \cN$ & $(d,r)$ & $L_{k^\natural}(\fg^\natural)$ & $X_{L_{k^\natural}(\fg^\natural)} \cap \cN^\natural$ & $(d^\natural,r^\natural)$ & sing.
	\\ \hline
	%--------------
	$E_6(a_1)$&$70$ & $9/7$ & $E_6(a_1)$ & $70$ & $\BC$ & $0$ & $0$ & $\ast$
	\\ \hline
	$D_5$&$68$ & $8/5$ & $D_5$ & $(68,1)$ & $L_{6/5}(\mathfrak{gl}_1)$ & $0$ & $(0,1)$ & $\ast$
	\\ \hline
	%$A_5$&$64$ & $9/5$ & $E_6(a_3)$ & $66$ & $L_{-2+3/10}(\fsl_2)$ & $[2]$ & $2$ & $A_1$
	%\\ \hline
	$D_4$ & $60$ & $9/4$ & $D_5(a_1)$ & $64$ & $L_{-3+3/2}(\fsl_3)$ & $[2,1]$ & $4$ & $a_2$
	\\ \hline
	$A_4$ & $60$ & $8/3$ & $A_4+A_1$ & $(62,1)$ & $L_{2}(\fgl_1)\otimes L_{-2+2/3}(\fsl_2)$ & $0 \times [2]$ & $(2,1)$ & $A_1$
	\\ \hline
	$A_3$ & $52$ & $4$ & $D_4(a_1)$ & $(58,2)$ & $L_3(\fgl_1) \otimes L_{-3+2}(\fso_5)$ & ? & ?
	\\ \hline
	$A_2+2A_1$ & $50$ & $9/2$ & $A_2+2A_1$ & $50$ & $L_0(\fgl_1) \otimes L_{-2+2}(\fsl_2)$ & $0 \times 0$ & $0$ & $\ast$
	\\ \hline
	$A_2$ & $42$ & $9/2$ & $A_2+2A_1$ & $50$ & $L_{-3+3/2}(\fsl_3)\otimes L_{-3+3/2}(\fsl_3)$ & $[2,1]\times [2,1]$ & $8$ & 
	\\ \cline{3-9}
	& & $6$ & $A_2$ & $42$ & $L_{-3+3}(\fsl_3) \otimes L_{-3+3}(\fsl_3)$ & $0 \times 0$ & $0$ & $\ast$
	\\ \hline
	$3A_1$ & $40$ & $6$ & $A_2$ & $42$ & $L_{-2+3/2}(\fsl_2)\otimes L_{-3+3}(\fsl_3)$ & $[2] \times 0$ & $2$ & $A_1$
	\\ \hline
	$2A_1$ & $32$ %& $9/2$ & $A_2+2A_1$ & $50$ & $L_{-9/2}(\fgl_1)\otimes L_{-5+3/2}(\fso_7)$ & ? & (?,$\ge1$)
	%\\ \cline{3-9}
	%&  
	& $6$ & $A_2$ & $42$ &  $L_{0}(\fgl_1)\otimes L_{-5+3}(\fso_7)$  & $0 \times [3,1^4]$ & $10$
	\\ \hline 
	$A_1$ & $22$ & $9$ & $A_1$ & $22$ & $L_{-6+6}(\fsl_6)$ & $0$ & $0$ & $\ast$
	\\ \cline{3-9}
	&  & $8$ & $2A_1$ & $(32,1)$ & $L_{-6+5}(\fsl_6)$ & $[2,1^4]$ & $(10,1)$ & $a_5$
	\end{tabular}
	}
	\vspace{3\baselineskip}
	\caption{Collapsing non-admissible levels in $E_7$.}\label{tbl:E7}
	\resizebox{\textwidth}{!}{
	%{\small
	\begin{tabular}{c|c|c||c|c||c|c|c||c}
		$f$ & $\dim \BO_f$ & $k+\check \Bh$ & $X_{L_k(\fg)} \cap \cN$ & $(d,r)$ & $L_{k^\natural}(\fg^\natural)$ & $X_{L_{k^\natural}(\fg^\natural)} \cap \cN^\natural$ & $(d^\natural,r^\natural)$ & sing.
		\\ \hline
		%-----------------
		$E_7(a_1)$ & $124$ & $14/11$ & $E_7(a_1)$ & $124$ & $\BC$ & $0$ & $0$ & $\ast$
		\\ \hline
		%$E_7(a_2)$ & $122$ & $14/9$ & $E_7(a_2)$ & $122$ & $\BC$ & $0$ & $0$ & $\ast$
		%\\ \hline
		$E_6$ & $120$ & $14/9$ & $E_7(a_2)$ & $122$ & $L_{-2+2/3}(\fsl_2)$ & $[2]$ & $2$ & $A_1$
		\\ \hline
		$E_6(a_1)$ & $118$ & $9/4$ & $E_6(a_1)$ & $(118,1)$  & $L_{3/2}(\fgl_1)$ & $0$ & $(0,1)$ & $\ast$
		\\ \hline
		$D_6$ & $118$ & $7/4$ & $E_7(a_3)$ & $(120,1)$ & $L_{-2+1/4}(\fsl_2)$ & $[2]$ & $(2,1)$ & $A_1$
		\\ \hline
		$A_6$ & $114$ & $7/3$ & $E_7(a_4)$ & $(116,1)$ & $L_{-2+1/3}(\fsl_2)$ & $[2]$ & $(2,1)$ & $A_1$
		\\ \hline
		$D_5$ & $112$ & $7/3$ & $E_7(a_4)$ & $(116,1)$ & $L_{-2+1/3}(\fsl_2)\otimes L_{-2+2/3}(\fsl_2)$ & $[2]\times [2]$ & $(4,1)$ & $\cN_{2A_1}$ 
		\\ \hline
		$(A_5)'$ & $108$ & $10/3$ & $E_6(a_3)$ & $110$ & $L_{-2+5/6}(\fsl_2) \otimes L_{-2+2}(\fsl_2)$ & $[2] \times 0$ & $2$ & $A_1$
		\\ \hline
		$A_4+A_2$ & $106$ & $13/3$ & $A_4+A_2$ & $106$ & $L_{-2+13}(\fsl_2)$ & $0$ & $0$ & $\ast$
		\\ \cline{3-9}
		& & $7/2$ & $D_5(a_1)+A_1$ & $(108,1)$ & $L_{-2+1/2}(\fsl_2)$ & $[2]$ & $(2,1)$ & $A_1$
		\\ \hline
		$(A_5)''$ & $102$ & $14/5$ & $E_7(a_5)$ & $112$ & $L_{-4+4/5}(\fg_2)$ & $G_2(a_1)$ & $10$ 
		\\ \hline
		$A_3+A_2+A_1$ & $100$ & $ 14/3$ & $A_3+A_2+A_1$ & $100$ & $L_{-2+2}(\fsl_2)$ & $0$ & $0$ & $\ast$
		\\ \cline{3-9}
		& & $11/2$ &  $A_3+A_2+A_1$ & $100$ & $L_{-2+22}(\fsl_2)$ & $0$ & $0$ & $\ast$
		\\ \hline
		$A_4$ & $100$ & $9/2$ & $A_4+A_1$ & $(104,1)$ & $L_{27}(\fgl_1)\otimes L_{-3+3/2}(\fsl_3)$ & $0 \times [2,1]$ & $(4,1)$ & $a_2$
		\\ \hline
		$D_4$ & $96$ & $7/2$ & $D_5(a_1)+A_1$ & $(108,1)$ & $L_{-4+3/2}(\fsp_6)$ & $[2^3]$ & $(12,1)$
		\\ \hline
		$A_2+3A_1$ & $84$ & $15/2$ & $A_2+3A_1$ & $84$ & $L_{-4+5}(\fg_2)$ & $0$ & $0$ & $\ast$
		\\ \cline{3-9}
		& & $6$ & $D_4(a_1)$ & $94$ & $L_{-4+2}(\fg_2)$ & $G_2(a_1)$ & $10$
		\\ \hline
		$2A_2$ & $84$ & $6$ & $D_4(a_1)$ & $94$ & $L_{-2+2}(\fsl_2)\otimes L_{-4+2}(\fg_2)$ & $0 \times G_2(a_1)$ & $10$ 
		\\ \hline
		$A_3$ & $84$ & $14/3$ & $A_3+A_2+A_1$ & $100$ & $L_{-2+2/3}(\fsl_2)\otimes L_{-5+5/3}(\fso_7)$ & $[2]\times[3^2,1]$ & $16$ 
		\\ \cline{3-9}
		& & $6$ & $D_4(a_1)$ & $94$ & $L_{-2+2}(\fsl_2)\otimes L_{-5+3}(\fso_7)$ & $0 \times [3,1^4]$ & $10$ 
		\\ \hline
		$A_2+2A_1$ & $82$ & $7$ & $A_2+3A_1$ & $(84,1)$ & $L_{-2+1}(\fsl_2)\otimes L_{-2+2}(\fsl_2)\otimes L_{-2+2}(\fsl_2)$ & $[2] \times 0 \times 0$ & $(2,1)$ & $A_1$
		\\ \cline{3-9}
		& & $15/2$ & $A_2+3A_1$ & $84$ & $L_{-2+3/2}(\fsl_2)\otimes L_{-2+3}(\fsl_2)\otimes L_{-2+5}(\fsl_2)$ & $[2] \times 0 \times 0$ & $2$ & $A_1$
		\\ \cline{3-9}
		& & $8$ & $A_2+2A_1$ & $82$ & $L_{-2+2}(\fsl_2)\otimes L_{-2+4}(\fsl_2)\otimes L_{-2+8}(\fsl_2)$ & $0 \times 0 \times 0$ & $0$ & $\ast$
		\\ \hline
		$A_2$ & $66$ & $7$ & $A_2+3A_1$ & $(84,1)$ & $L_{-6+3}(\fsl_6)$ & $[2^3]$ & $(18,1)$ 
		\\ \cline{3-9}
		& & $9$ & $A_2+A_1$ & $(76,1)$ & $L_{-6+5}(\fsl_6)$ & $[2,1^4]$ & $(10,1)$ & $a_5$
		\\ \hline
		$(3A_1)'$ & $64$ & $10$ & $A_2$ & $66$ & $L_{-2+5/2}(\fsl_2)\otimes L_{-4+4}(\fsp_6)$ & $[2]\times 0$ & $2$ & $A_1$
		\\ \hline
		$(3A_1)''$ & $54$ & $7$ & $A_2+3A_1$ & $(84,1)$ & $L_{-9+4}(\ff_4)$ & ? & ? 
		\\ \hline
		$2A_1$ & $52$ & $12$ & $2A_1$ & $52$ & $L_{-2+4}(\fsl_2)\otimes L_{-7+7}(\fso_9)$ & $0 \times 0$ & $0$ & $\ast$
		\\ \cline{3-9}
		&  & $10$ & $A_2$ & $66$ & $L_{-2+2}(\fsl_2)\otimes L_{-7+5}(\fso_9)$ & $0 \times [3,1^6]$ & $14$
		\\ \hline
		$A_1$ & $34$ & $12$ & $2A_1$ & $52$ & $L_{-10+8}(\fso_{12})$ & $[2^2, 1^8]$ & $18$ & $d_6$
		\\ \cline{3-9}
		&& $14$ & $A_1$ & $34$ & $L_{-10+10}(\fso_{12})$ & $0$ & $0$ & $\ast$
	\end{tabular}
	}
\end{table}

\begin{table}[h]
	\caption{Collapsing non-admissible levels in $E_8$.}\label{tbl:E8}
\resizebox{\textwidth}{!}{
%{\small
	\begin{tabular}{c|c|c||c|c||c|c|c||c}
		$f$ & $\dim \BO_f$ & $k+\check \Bh$ & $X_{L_k(\fg)} \cap \cN$ & $(d,r)$ & $L_{k^\natural}(\fg^\natural) $ & $X_{L_{k^\natural}(\fg^\natural)}\cap \cN^\natural$ & $(d^\natural,r^\natural)$ & sing.
		\\ \hline
		%------------------
		$E_8(a_1)$ & $238$ & $24/19$ & $E_8(a_1)$ & $238$ & $\BC$ & $0$ & $0$ & $\ast$
		\\ \hline
		$E_8(a_2)$ & $236$ & $20/13$ & $E_8(a_2)$ & $236$ & $\BC$ & $0$ & $0$ & $\ast$
		\\ \hline
		$E_8(a_4)$ & $232$ & $15/7$ & $E_8(a_4)$ & $232$ & $\BC$ & $0$ & $0$ & $\ast$
		\\ \hline
		$E_7$ & $232$ & $12/7$ & $E_8(a_3)$ & $234$ & $L_{-2+3/14}(\fsl_2)$ & $[2]$ & $2$ & $A_1$
		\\ \hline
		$E_7(a_1)$ & $228$ & $24/11$ & $E_8(b_4)$ & $230$ & $L_{-2+2/11}(\fsl_2)$ & $[2]$ & $2$ & $A_1$
		\\ \cline{3-9}
		& & $20/9$ & $E_8(b_4)$ & $230$ & $L_{-2+2/9}(\fsl_2)$ & $[2]$ & $2$ & $A_1$
		\\ \hline
		$D_7$ & $226$ & $18/7$ & $E_8(a_5)$ & $228$ & $L_{-2+9/14}(\fsl_2)$ & $[2]$ & $2$ & $A_1$
		\\ \hline
		$A_7$ & $218$ %& $24/7$ & $E_8(b_6)$ & $220$ & $L_{-2+3/14}(\fsl_2)$ & $[2]$ & $2$ & $A_1$
		%\\ \cline{3-9}
		%& 
		& $ 15/4$ & $E_8(b_6)$ & $220$ & $L_{-2+3/2}(\fsl_2)$ & $[2]$ & $2$ & $A_1$
		\\ \hline
		$E_6$ & $216$ & $8/3$ & $E_8(b_5)$ & $(226,1)$ & $L_{-4+2/3}(G_2)$ & ? & ? 
		\\ \hline
		$D_6$ & $216$ & $10/3$ & $E_8(a_6)$ & $224$ & $L_{-3+5/6}(\fso_5)$ & $[5]$ & $8$ & $\cN_{C_2}$
		\\ \hline
		$E_6(a_1)$ & $214$ & $24/7$ & $E_8(b_6)$ & $220$ & $L_{-3+3/7}(\fsl_3)$ & $[3]$ & $6$ & $\cN_{A_2}$
		\\ \cline{3-9}
		& &  $15/4$ & $E_8(b_6)$ & $220$ & $L_{-3+3/4}(\fsl_3)$ & $[3]$ & $6$ & $\cN_{A_2}$
		\\ \hline
		$D_6(a_2)$ & $204$ & $6$ & $E_8(a_7)$ & $208$ & $L_{-2+3/2}(\fsl_2)\otimes L_{-2+3/2}(\fsl_2)$ & $[2]\times [2]$ & $4$ 
		\\ \hline
		$D_5$ & $200$ & $4$ & $D_7(a_2)$ & $(216,2)$ & $L_{-5+1}(\fso_7)$ & ? & ? 
		\\ \hline
		$A_4+A_2$ & $194$ & $20/3$ & $A_4+A_2+A_1$ & $196$ & $L_{-2+2}(\fsl_2)\otimes L_{-2+2/3}(\fsl_2)$ & $0 \times [2]$ & $2$ & $A_1$
		\\ \hline
		$D_4(a_1)+A_2$ & $184$ & $25/3$ & $D_4(a_1)+A_2$ & $184$ & $L_{-3+5}(\fsl_3)$ & $0$ & $0$ & $\ast$
		\\ \cline{3-9}
		& & $9$ & $D_4(a_1)+A_2$ & $184$ & $L_{-3+8}(\fsl_3)$ & $0$ & $0$ & $\ast$
		\\ \hline
		$A_3+A_2+A_1$ & $182$ & $19/2$ & $A_3+A_2+A_1$ & $182$ & $L_{-2+2}(\fsl_2)\otimes L_{-2+38}(\fsl_2)$ & $0 \times 0$ & $0$ & $\ast$
		%\\ \cline{3-9}
		%& & $8$ & $D_4(a_1)+A_2$ & $(184,1)$ & $L_{-2+1/2}(\fsl_2)\otimes L_{-2+2}(\fsl_2)$ & $[2] \times 0$ & $(2,1)$ & $A_1$
		\\ \hline
		$A_4$ & $180$ & $20/3$ & $A_4+A_2+A_1$ & $196$ & $L_{-5+5/3}(\fsl_5)$ & $[3,2]$ & $16$ 
		\\ \cline{3-9}
		& & $15/2$ & $A_4+2A_1$ & $192$ & $L_{-5+5/2}(\fsl_5)$ & $[2^2,1]$ & $12$ 
		\\ \hline
		$D_4$ & $168$ & $6$ & $E_8(a_7)$ & $208$ & $L_{-9+3}(F_4)$ & NA & NA 
		\\ \hline
		$2A_2$ & $156$ & $12$ & $2A_2$ & $156$ & $L_{-4+4}(\fg_2) \otimes L_{-4+4}(\fg_2)$ & $0 \times 0$ & $0$ & $\ast$
		\\ \hline
		$A_3$ & $148$ & $8$ & $D_4(a_1)+A_2$ & $(184,1)$ & $L_{-9+3}(\fso_{11})$ & NA & NA 
		\\ \cline{3-9}
		& & $10$ & $D_4(a_1)+A_1$ & $176$ & $L_{-9+5}(\fso_{11})$ & NA & NA 
		\\ \hline
		$A_2+2A_1$ & $146$ & $12$ & $2A_2$ & $156$ & $L_{-2+2}(\fsl_2)\otimes L_{-5+3}(\fso_7)$ & $0 \times [3,1^4]$ & $10$ 
		%\\ \cline{3-9}
		%& & $25/2$ & $A_2+3A_1$ & $154$ & $L_{-2+5}(\fsl_2)\otimes L_{-5+7/2}(\fso_7)$ & $0 \times [2^2,1^3]$ & $8$ & $b_3$
		\\ \cline{3-9}
		& & $14$ & $A_2+2A_1$ & $146$ & $L_{-2+14}(\fsl_2)\otimes L_{-5+5}(\fso_7)$ & $0 \times 0$ & $0$ & $\ast$
		\\ \hline
		$4A_1$ & $128$ & $15$ & $A_2+A_1$ & $136$ & $L_{-5+9/2}(\fsp_8)$ & $[2,1^6]$ & $8$ & $c_4$
		\\ \hline
		$A_2$ & $114$ & $12$ & $2A_2$ & $156$ & $L_{-12+6}(\fe_6)$ & $A_2$ & $42$ 
		\\ \cline{3-9}
		& & $15$ & $A_2+A_1$ & $136$ & $L_{-12+9}(\fe_6)$ & $A_1$ & $22$ & $e_6$
		\\ \hline
		$3A_1$ & $112$ & $18$ & $A_2$ & $114$ & $L_{-2+9/2}(\fsl_2)\otimes L_{-9+9}(\ff_4)$ & $[2] \times 0$ & $2$ & $A_1$
		\\ \hline
		$2A_1$ & $92$ & $18$ & $A_2$ & $114$ & $L_{-11+9}(\fso_{13})$ & $[3,1^{10}]$ & $22$ 
		\\ \cline{3-9}
		& & $20$ & $2A_1$ & $92$ & $L_{-11+11}(\fso_{13})$ & $0$ & $0$ & $\ast$
		\\ \hline
		$A_1$ & $58$ & $20$ & $2A_1$ & $92$ & $L_{-18+14}(\fe_7)$ & $A_1$ & $34$ & $e_7$
		\\ \cline{3-9}
		& & $24$ & $A_1$ & $58$ & $L_{-18+18}(\fe_7)$ & $0$ & $0$ & $\ast$
	\end{tabular}
}
\end{table}

\printbibliography
\listoftables
\end{document}